\documentclass{amsart}
\usepackage[T1]{fontenc}   

\usepackage{amssymb}
\usepackage{amssymb,amsthm} 
\usepackage{amscd,amssymb,graphics}

\usepackage[usenames]{color}
\usepackage{mathrsfs}

\usepackage{color}
\usepackage{amsfonts}
\usepackage{amsmath}
\usepackage{amsxtra}
\usepackage{latexsym}
\usepackage[mathcal]{eucal}
\usepackage{enumerate}
\usepackage{ifsym}
\usepackage{yfonts}
\usepackage{calc}
\usepackage{tikz}
\usepackage{MnSymbol}
\usepackage[numbers]{natbib}

\newtheorem{theorem}{Theorem}[section]
\newtheorem{fact}[theorem]{Fact}
\newtheorem{corollary}[theorem]{Corollary}
\newtheorem{lemma}[theorem]{Lemma}
\newtheorem{proposition}[theorem]{Proposition}

\newtheorem{question}[theorem]{Question}

\theoremstyle{definition}
\newtheorem{definition}[theorem]{Definition}
\theoremstyle{remark}
\newtheorem{remark}[theorem]{Remark}
\newtheorem{example}[theorem]{Example}

\newcommand{\ben}{\begin{enumerate}}
\newcommand{\een}{\end{enumerate}}
\newcommand{\bit}{\begin{itemize}}
\newcommand{\eit}{\end{itemize}}

\def\R {{\mathbb R}}
\def\Q {{\mathbb Q}}

\def\C {{\mathbb C}}
\def\N{{\mathbb N}}
\def\T{{\mathbb T}}

\def\Z {{\mathbb Z}}

\def\F{{\mathbb F}}

\def\H{{\mathcal H}}

\def\QED{\nobreak\quad\ifmmode\roman{Q.E.D.}\else{\rm Q.E.D.}\fi}

\def\H(L)M{hereditarily  (locally) minimal}
\def\D(L)M{densely (locally) minimal}
\def \HM {hereditarily   minimal}

\def\DM{densely minimal}
\def\CM{c-minimal}
\def\CTM{c-totally minimal}

\def \HM {hereditarily   minimal}

\def\PSL{\operatorname{PSL}}

\def\SL{\operatorname{SL}}
\def\soc{\operatorname{soc}}
\DeclareMathOperator{\SO}{SO}
\DeclareMathOperator{\ISO}{Iso}

\begin{document}

\title[]{C-Minimal topological groups}

\author[]{Wenfei Xi}
\address[W. Xi]{\hfill\break
	School of Applied Mathematics
	\hfill\break
	Nanjing University of Finance \& Economics
	\hfill\break
	Wenyuan Road No. 3, 210046 Nanjing
	\hfill\break
	China}
\email{xiwenfei0418@outlook.com}

\author[]{Menachem Shlossberg}
\address[M. Shlossberg]
{\hfill\break Mathematics Unit
	\hfill\break Shamoon College of Engineering
	\hfill\break 56 Bialik St., Beer-Sheva 84100
	\hfill\break Israel}
\email{menacsh@sce.ac.il}

\keywords{{Lie group, minimal group, c-minimal group, hereditarily non-topologizable group, locally solvable group}}

\date{June 27, 2021}
\subjclass[2010]{22A05, 22E20, 22E25, 54H11}

\begin{abstract}
We study topological groups having all closed subgroups (totally) minimal and we call such groups {\it c-(totally) minimal}. We show that a  locally compact c-minimal connected group is compact. Using a well-known theorem of Hall and  Kulatilaka \cite{HK} and a characterization of a certain class of Lie groups, due to Grosser and Herfort \cite{GH}, we prove that a \CM\ locally solvable Lie group is compact.
	
It is  shown that if a topological group  $G$ contains a compact open normal subgroup $N$, then $G$  is \CTM\ if and only if  $G/N$ is hereditarily non-topologizable. Moreover, a \CTM\ group that is either complete solvable or strongly compactly covered must be compact. Negatively  answering  \cite[Question 3.10(b)]{DMe} of Dikranjan and Megrelishvili we find, in contrast, a totally minimal solvable (even metabelian) Lie group that is not compact.
We also prove that the group $A\times F$ is c-(totally) minimal for every (respectively, totally) minimal abelian group $A$ and every finite group $F.$
\end{abstract}

\maketitle
\section{Introduction}
All topological groups in this paper are Hausdorff. A topological group $(G, \tau)$ is called {\it minimal}  \cite{Doitch,S71} if there exists no group topology on $G$ that is strictly coarser than $\tau.$  Equivalently, if every continuous isomorphism $f : G\to H,$ where $H$ is an arbitrary topological group,  is a topological isomorphism.  If the same holds for every quotient of $G$, then $G$ is {\it totally minimal} \cite{DP}. This is exactly a group $G$ that satisfies the open mapping theorem. The topological semidirect product $\R \leftthreetimes \R_{+}$, where the multiplicative group of positive reals  $\R_{+}$ acts on $\R$ by multiplication,  is a locally compact  minimal group that is not totally minimal (see \cite{DS79}).
For more information on minimal groups we refer the reader to  \cite{B,DS,EDS,P} (see also the survey \cite{DMe} and the book \cite{DPS}).

\vskip 0.3cm
In 1971, Stephenson proved that a minimal locally compact abelian group is compact (\cite[Theorem 1]{S71}). This result was extended substantially by Prodanov and Stoyanov.

\begin{fact}\label{fac:maispre} {\normalfont\cite{PS}}
Every minimal abelian group is precompact.
\end{fact}

Recently, Banakh \cite{Banakh} provided a quantitative generalization of this theorem.

\vskip 0.3 cm
Dikranjan, Toller and the authors introduced the following two concepts in \cite[Definition 1.2]{firstpaper}\label{hm} and \cite[Definition 1.5]{secondpaper}, respectively.

\begin{definition}\label{defhmdm}
A topological group $G$ is said to be:\ben \item
{\em hereditarily minimal} if every subgroup of $G$ is minimal;
\item {\em densely minimal} if every  dense subgroup of $G$ is minimal. \een
\end{definition}

It is easy to see that a topological group $G$  is hereditarily minimal if and only if every closed subgroup of $G$ is densely minimal. Prodanov \cite{P} proved that  the $p$-adic integers are the only infinite (locally) compact hereditarily minimal abelian groups. Later, Dikranjan and  Stoyanov \cite{DS} classified all hereditarily minimal abelian groups. In \cite[Theorem D]{firstpaper}\label{thm:hmabc}, all infinite locally compact solvable \HM\ groups  were classified.

\vskip 0.3cm
The next theorem provides another extension of Prodanov's theorem.

\begin{fact}{\normalfont\cite[Theorem C]{secondpaper}}\label{HMDMZP}
For an infinite locally compact abelian group $K$, the following conditions are equivalent:
\ben[(a)]
\item $K$ is a \HM \ group;
\item $K$ is a \DM \ group;
\item $K$ is isomorphic to $\Z_p$ for some prime $p.$\een
\end{fact}

In this paper, we study the topological groups having all {\em closed} subgroups minimal.

\begin{definition}\label{defcm}
A topological group $G$ is called  {\em c-minimal} if every closed subgroup of $G$ is minimal.
\end{definition}

 Clearly, a compact group is c-minimal. It follows from Fact \ref{fac:maispre} that every closed abelian subgroup  of a complete c-minimal  group $G$ (e.g., the center $Z(G)$)  is compact.
Megrelishvili \cite{MEG08}  proved that every topological group is a group retract of a minimal group. Hence, every topological group is a closed subgroup (and a quotient) of a minimal group.

Note that  both the class of minimal groups and the class of totally minimal groups    are closed under taking closed central subgroups  (see \cite[Proposition 7.2.5]{DPS}). In particular, a minimal abelian group is c-minimal  while this property does not hold even for a two-step nilpotent minimal group, as the next example shows.

\begin{example}\label{nilpo2}
The  Weyl-Heisenberg group $G=H(\R)/Z(H(\R))$, where
\[
H(\R)=\Bigg\{\left(\begin{array}{ccc}
1 & a &  b \\
0 & 1 & c \\
0 & 0 & 1
\end{array}\right)\bigg | \ \ a,b,c \in \R \Bigg\}
\]
is the classical Heisenberg group, is minimal (see \cite[Theorem~5.11]{DMe}). However, $G$ is not c-minimal since it contains a copy of $\R$ as a closed non-minimal subgroup.
\end{example}

The following diagram  summarizes the  interrelations between the classes of minimal groups defined above.
Example \ref{firstdia} shows, among other things, that all inclusions are proper.

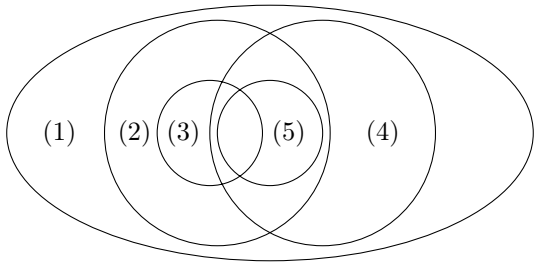
\begin{center}
\begin{tikzpicture}
\draw
(-0.8,0) circle (0.7) (0,-1.3)  node at (-1.15,0) {(3)}
(-0.7,0) circle (1.5) (-2.5,-0.5)  node at (-1.8,0) {(2)}
(0.7,0) circle (1.5) (2.5,-0.5)  node at (1.5,0) {(4)}
(0,0) circle (0.7) (0,-1.3)  node at (0.25,0) {(5)}
(90:0.0cm) ellipse (3.5cm and 1.7cm) node at (-2.8,0) {(1)};
\end{tikzpicture}
\end{center}

(1) minimal groups;

(2) c-minimal groups;

(3) compact groups;

(4) densely minimal groups;

(5) hereditarily minimal groups.

\begin{example}\label{firstdia}
\ben [(a)]\
\item
By Bader and Gelander \cite{BG}, the special linear group $\SL(n, \F)$ is totally minimal for every local field $\F$ (see also \cite{MS} for a new independent proof). In particular, $\SL(3, \C)$ is totally minimal. By \cite[Corollary 5.11(3)]{MS}, its dense subgroup $\SL(3, \Q(i))$  fails to be minimal, where $\Q(i):=\{a+bi: a,b\in\Q\}$ is the Gaussian rational field. Moreover, $\SL(3, \C)$ is not c-minimal as it contains $\C$ as a closed subgroup. This shows that the union of (2) and (4) is properly contained in (1).

\item
The rational circle $\Q/\Z$ is minimal precompact abelian group (see \cite[Example 1.1(a)]{DMe}). Being abelian it is also c-minimal. However, this group is neither compact nor hereditarily minimal  (see [11] for the classification of all torsion abelian hereditarily minimal groups).  This shows that the union of (3) and (5) is properly contained in (2). Using Example \ref{example:NHMCM} below one can find  a complete non-compact c-minimal group that is not \HM.
	
\item
Let $G$ be any infinite compact abelian group  that is not isomorphic to $\Z_p$ for any prime $p$. By Fact \ref{HMDMZP}, $G$ is in (3) but not in (4).
	
\item
Any non-closed subgroup of $\Z_p$ is  in (5) but not in (3).

\item
By \cite[Theorem 4.7(a)]{Meg}, the group $K \leftthreetimes K^*$ is  minimal for every non-discrete locally retrobounded division ring  $K$, where $K^* = (K\setminus\{0\},\cdot)$ acts on $K$ by multiplication. In particular, the group $G=(\Q_p,+)\leftthreetimes \Q_p^*$ is minimal.  By \cite[Corollary 2.11]{firstpaper}, $G$ is even densely minimal. However, it is not c-minimal since it contains $\Q_p$ as a closed non-minimal subgroup. This means that (4) is not contained in (2).

\item
By  \cite[Example 3.11]{secondpaper}, the compact two-step nilpotent group
\[
G=\Bigg\{\left(\begin{array}{ccc}
1 & a &  b \\
0 & 1 & c \\
0 & 0 & 1
\end{array}\right)\bigg | \ \ a,b,c \in \Z_p\Bigg\}
\]
is densely minimal but not hereditarily minimal.  Since  every compact group is $c$-minimal, we deduce that a topological group that is simultaneously $c$-minimal and densely minimal  need not be hereditarily minimal. So (5) is properly contained in the intersection of (2) and (4).

\item
By \cite[Corollary 1.12]{firstpaper}, every infinite hereditarily minimal locally compact solvable group $G$ is compact metabelian. In particular, it lies in the intersection of (3) and (5).
\een
\end{example}

\vskip 0.5cm

\subsection{Notation and terminology}
The fields of rationals, reals and complex numbers   are denoted  by $\mathbb{Q},\mathbb{R} $ and  $\mathbb{C}$, respectively, and $\T=\{a\in \C\ : \ |a|=1\}$ denotes the unit circle group.
For a prime number $p$, $\Z_p$ stands for the ring of $p$-adic integers and $\Q_p$ is the field of $p$-adic numbers. We denote by $\Z$ the group of integers, while $\N$ and $\N_{+}$ are its subsets of non-negative integers and positive natural numbers, respectively.

An element  $x$  of a group $G$ is \emph{torsion} if  the subgroup of $G$ generated by $x$,  denoted by  $\langle x\rangle$, is finite. Moreover, $G$ is \emph{torsion} if every element of $G$ is torsion. Let $\mathcal P$ be an algebraic (or set-theoretic) property.  A group is called \emph{locally $\mathcal P$} if every finitely generated subgroup has the property $\mathcal P$. For example, in a locally finite group every finitely generated subgroup is finite. In particular, a locally finite group is torsion.

A group $G$ is \emph{solvable} if  there exist $k\in \N_{+}$ and a subnormal series $$G_0=\{e\}\unlhd G_1\unlhd \cdots \unlhd G_k=G,$$  where $e$ denotes the identity element, such that the quotient group $G_j/G_{j-1}$  is abelian for every $j\in \{1,\ldots, k\}.$ In particular, $G$ is  {\em metabelian}, if $k\leq 2.$
The subgroup $Z(G)$ denotes the \emph{center} of $G,$ and we set $Z_0(G) = \{e\}$ and  $Z_1(G) = Z(G)$.  For $n > 1$, the  \emph{$n$-th center} $Z_n(G)$ is defined as follows:
$$Z_n(G)=\{x\in G: [x,y]\in Z_{n-1}(G)  \text{ for every } y\in G\},$$
where $[x,y]$ denotes the commutator $xyx^{-1}y^{-1}$.  A group is {\em nilpotent} if $Z_n(G) = G$ for some $n\in \N$. In this case, its nilpotency class is the minimum of such $n$. We denote by $G'$ the {\em derived subgroup} of $G$, namely, the subgroup of $G$ generated by all commutators $[x,y]$, where $x,y\in G$.

Let $G$ be a topological group and $H$ be its subgroup. The closure of $H$ in $G$ is denoted by $\overline{H}$, and $c(G)$ is the connected component of $G$. If  $c(G)=\{e\},$ then $G$ is {\em totally disconnected}.  It is known that  the connected component of  a Lie group  is open. A  topological group is called  \emph{complete} if it is complete with respect to its two-sided uniformity, and it is \emph{precompact} if its completion is compact. Finally,  $G$ is \emph{balanced} if it admits a local base at the identity consisting of  neighborhoods invariant under conjugations.

\section{c-minimal Lie groups}
The minimality of Lie groups has been studied by many authors (see \cite{Est,Goto,Omori,RS}).
By Omori \cite{Omori},  a connected nilpotent Lie group  with compact center is minimal.
By Remus and Stoyanov \cite{RS}, a connected semi-simple Lie group is totally minimal if and only if its center is finite. In particular, the special linear groups $\SL(n,\R)$ are totally minimal.

The main goal of this section is to prove that  a \CM\ locally solvable Lie group is compact (see Theorem \ref{comp} below).  We start by recalling that a topological group $G$ is  {\it compactly covered} if every element $g\in G$ is contained in a compact subgroup of $G$. It is easy to see that the class of compactly covered groups is closed under taking quotients and closed subgroups.

\begin{lemma}\label{lem:ccov}
A complete c-minimal group $G$ is compactly covered.
\end{lemma}

\begin{proof}
Take a non-torsion  element $g\in G$. As $G$ is  c-minimal group, it follows from Fact \ref{fac:maispre} that $\overline{\langle g\rangle}$ is precompact. The completeness of $G$ implies that its subgroup   $\overline{\langle g\rangle}$ is compact. This proves that $G$ is compactly covered.
\end{proof}

As noted above, a compact group is c-minimal.  Using the next reformulated result, we shall  prove in Proposition \ref{prop:cgcomp} that the converse is true for connected locally compact groups.

\begin{fact}{\normalfont\cite[Lemma 2.15]{BWY}}\label{fac:cc}
Let $H$ be a connected locally compact group. If $H$ is compactly covered, then it is compact.
\end{fact}

\begin{proposition}\label{prop:cgcomp}
If $G$ is a locally compact $c$-minimal group, then $c(G)$ is compact.  In particular, a connected locally compact $c$-minimal group   is compact.
\end{proposition}
\begin{proof}
By Lemma \ref{lem:ccov}, the closed subgroup $c(G)$ is compactly covered. Being also locally compact and connected, it must be compact by Fact \ref{fac:cc}. The last assertion now trivially follows.
\end{proof}

\begin{lemma}\label{lem:liebal}
Let $G$ be  a locally compact  group that is either balanced or Lie. If $G$ is c-minimal, then it contains a compact open normal subgroup.
\end{lemma}
\begin{proof}	
Let $G$ be a locally compact  c-minimal  group. By Proposition \ref{prop:cgcomp}, the connected component $c(G)$ is compact. In case $G$ is Lie, then its compact normal subgroup  $c(G)$  is open.

In case $G$ is balanced, then $G/c(G)$ is a locally compact  totally disconnected balanced group. Hence, the latter group has a local base at the identity consisting of compact open normal subgroups. Pick any $K$ from this local base. As $c(G)$ is compact normal subgroup of $G$, it follows that $N=q^{-1}(K)$ is a compact open normal subgroup of $G$, where  $q: G \to G/c(G)$ is the quotient homomorphism.
\end{proof}

\begin{proposition}\label{corollary:abfi}
Let $G$ be a c-minimal Lie group. Then $G$ satisfies the following properties:
\ben
\item $c(G)$ is compact;
\item every abelian subgroup of $G/c(G)$ is finite.
\een
\end{proposition}
\begin{proof}
By \cite[Theorem 3.2]{GH},  a Lie group $G$ satisfies properties $(1)$-$(2)$ if and only if all of its abelian  subgroups are relatively compact. Now, if $H$ is an abelian subgroup of a c-minimal Lie group $G$, then $\overline{H}$ is minimal.  Being locally compact minimal abelian, $\overline{H}$ must be compact in view of Fact \ref{fac:maispre}.
\end{proof}

The next example, provided by D. Dikranjan and simplified on the advice of  the referee,  shows that a Lie group satisfying properties $(1)$-$(2)$ of Proposition \ref{corollary:abfi}  need not be minimal.  Recall that a group  is called {\it topologizable} if it admits a non-discrete Hausdorff group topology and {\it non-topologizable} otherwise.

\begin{example}\label{ex:dik}
Let $G=(T, \tau_d)$, where $T$ is a topologizable Tarski monster group (such groups can be found in \cite{KOO,nontopo}) equipped with the discrete topology $\tau_d$. Then $G$ is a Lie group with $c(G)=\{e\}$ compact, and all abelian subgroups of $G/c(G)=G$ are finite cyclic. However, $G$ is not minimal since $T$ is  topologizable.
\end{example}

The next fact, which deals with infinite discrete hereditarily minimal groups, will also be used in the sequel.

\begin{fact}\label{fac:dhm}{\normalfont\cite[Lemma 3.5]{firstpaper}}
If $G$ is an infinite discrete \HM\ group, then the abelian subgroups of $G$ are finite.
In particular, the center of $G$ is finite, $G$ is torsion but it is neither locally finite nor locally solvable.	
\end{fact}

By Fact \ref{fac:dhm}, a locally solvable discrete \HM\ group is finite. One of the key ingredients in the proof of Theorem \ref{comp}, which extends this result, is the following classical theorem of Hall and  Kulatilaka.

\begin{fact}{\normalfont\cite{HK}}\label{abinside}
Every infinite locally finite group contains an infinite abelian subgroup.
\end{fact}

\begin{theorem}\label{comp}
If $G$ is a \CM\ locally solvable Lie group, then $G$ is compact.
\end{theorem}
\begin{proof}
As a quotient of a locally solvable group is still locally solvable (see page 2 of \cite{DIX}), we deduce that $G/c(G)$ is locally solvable. By Proposition \ref{corollary:abfi}(2), the factor group $G/c(G)$ is torsion. So, it is locally finite according to \cite[Proposition 1.1.5]{DIX}.
Since $c(G)$ is compact by Proposition \ref{corollary:abfi}(1),  it suffices to prove that  $G/c(G)$ is finite. Assume for a contradiction that $G/c(G)$ is  infinite. In view of Fact \ref{abinside},  it  must contain an infinite abelian subgroup. This contradicts Proposition \ref{corollary:abfi}(2).
\end{proof}

\section{c-totally minimal groups}
\begin{definition}\label{def:ctm}
Call a topological group $G$ {\em c-totally minimal} if every closed subgroup of $G$ is totally  minimal.
\end{definition}

The main results of this section are Theorem \ref{pro:procmin} and Theorem \ref{thm:ballie}. Using Theorem \ref{pro:procmin}, one can produce many precompact groups  which are c-(totally) minimal and are not (necessarily)  abelian or compact. In Theorem \ref{thm:ballie}, we characterize the c-totally minimal groups having a compact open normal subgroup. We start by recalling some related concepts.

\begin{definition}\label{defhtmdtm}
A topological group $G$ is said to be:
\ben
\item \cite[Definition 6.15]{firstpaper} {\em hereditarily totally minimal} if every subgroup of $G$ is totally minimal;
\item \cite[Definition 1.8]{secondpaper} {\em densely totally minimal} if every  dense subgroup of $G$ is totally minimal.
\een
\end{definition}

Now we give  examples of  densely totally minimal groups that are not c-minimal.

\begin{example}\label{tmdtmctmhtm}
\ben \
\item
Consider the projective special linear group $\PSL(n, \R)= \SL(n, \R)/ Z$, where $\SL(n, \R)$ is the special linear group with center  $Z= \{c\operatorname{I}:  c^n=1\}$. By \cite[Theorem 2.4]{RS}, the group $\PSL(n, \R)$  is  totally minimal. Being also simple (see \cite[Corollary 3.2.9]{Robinson}),  it must be densely totally minimal according to \cite[Lemma 4.4]{secondpaper}. However, since $\PSL(n, \R)$ is a connected non-compact Lie group, it is not c-minimal by Proposition \ref{prop:cgcomp}.

\item
Shelah \cite{Sh} constructed a simple non-topologizable  group $G$ under the assumption of  CH.
Using the same arguments from (1), one can see that $G$ is  densely totally minimal when equipped with the (unique)  discrete topology. Being also torsion free, it is not hereditarily minimal  as discrete hereditarily minimal groups are torsion by Fact \ref{fac:dhm}. Moreover, since discrete c-minimal groups are hereditarily minimal, we deduce that $G$ is not c-minimal.
\een
\end{example}

Clearly, a compact group is c-totally minimal. A hereditarily totally minimal group is both densely totally minimal and c-totally minimal. As the following example shows, the converse is not true in general.

\begin{example}
By \cite[Example 4.6]{secondpaper},  the special orthogonal group $G=\SO(n, \R)$, where $n\geq 5$, is densely totally minimal but not hereditarily minimal. Being compact, $G$ is also c-totally minimal.
\end{example}

 The next fact was originally proved in \cite{EDS}  and we use it several times in the sequel.

\begin{fact}{\normalfont\cite[Theorem 7.3.1]{DPS}}\label{fac:EDS}
If a topological group $G$ contains a compact normal subgroup $N$ such that $G/N$ is (totally) minimal, then $G$ is (resp., totally) minimal.
\end{fact}

Recall  that a subset $H$ of a group $G$ is called {\it unconditionally closed} if $H$ is closed in every Hausdorff group topology of $G$ (see \cite{Markov}).
For example, the center $Z(G)$ is always unconditionally closed.

\begin{theorem}\label{pro:procmin}
Let $G=A\times F$ be a topological direct product, where $A$ is an abelian group and $F$ is a finite group.
\ben
\item If $A$ is minimal, then $G$ is c-minimal.
\item If $A$ is totally minimal, then $G$ is c-totally minimal.
\een
\end{theorem}
\begin{proof}
(1) Let $H$ be a closed subgroup of $G$. If $H$ is also abelian, then $H$ is a subgroup of $A\times p_2(H)$, where $p_2:H\to F$ is the canonical projection on the second coordinate of $H$. By Fact \ref{fac:EDS}, $A\times p_2(H)$ is  minimal since  $p_2(H)$ is a finite normal subgroup of  $A\times p_2(H)$ and $A\cong (A\times p_2(H))/p_2(H)$ is  minimal. Being a closed subgroup of a  minimal abelian group, $H$ is also minimal.
	
Now, let $H$ be a (not necessarily abelian) closed subgroup of $G.$ To prove that $H$ is minimal   let $\sigma\subseteq \tau|_H$ be a coarser Hausdorff group topology on $H$, where $\tau$ is the given product topology on $G.$
Since $Z(H)$ is unconditionally closed in $H$,  and $H$ is $\tau$-closed in $G$ we deduce that $Z(H)$ is a $\tau$-closed abelian subgroup of $G$. By the previous step, $Z(H)$ is minimal. In particular, $\sigma|_{Z(H)}=\tau|_{Z(H)}=(\tau|_H)|_{Z(H)}$. Since $Z(H)$ is unconditionally closed in $H$ and $$[H:Z(H)]\leq [H:Z(G)\cap H]=[HZ(G)/Z(G)]\leq [G/Z(G)]<\infty,$$ we deduce that the quotient topology on $H/Z(H)$ induced by either $\sigma$ or $\tau|_H$ is the discrete topology. By Merson's Lemma (see \cite[Lemma 7.2.3]{DPS}),  $\sigma=\tau|_H$ and we  establish the minimality of $H$.

(2) Let $H$ be a closed subgroup of $G=A\times F$ and suppose that $A$ is totally minimal and $F$ is finite.
As in the proof of (1), one can show that $M=A\times  p_2(H)$ is totally minimal using Fact \ref{fac:EDS}.  This implies that the  quotient group $M/M'$ is totally minimal. Being abelian, the latter group is even c-totally minimal. As $A$ is abelian, we deduce that $M'=H'$. Being   a closed subgroup of a c-totally minimal group, $H/M'=H/H'$ must be totally minimal. By Fact \ref{fac:EDS}, $H$ is totally minimal since $H'$ is finite.
\end{proof}

Note that by Fact \ref{fac:maispre}, the group $A$ from Theorem \ref{pro:procmin} must be precompact, while the group $F$ need not be abelian.

\begin{example}
\ben \
\item Let $A$ be a proper subgroup of $\Q/\Z$ containing  $\soc(\Q/\Z)$, where the socle  $\soc(\Q/\Z)$ is the subgroup generated by all  subgroups of $\Q/\Z$ of prime order. By \cite[Example 3.7]{DMe}, $A$ is minimal but not totally minimal. By Theorem \ref{pro:procmin}(1), $A\times F$ is
c-minimal for every finite group $F.$
\item Consider the topological group $A=(\Z,\tau_p)$, where $\tau_p$ is the p-adic topology for some prime $p$. It is known that $A$ is totally minimal (see \cite[Example 4.3.5]{DPS}).  Theorem \ref{pro:procmin}(2) implies that $A\times F$ is
c-totally minimal for every finite group $F.$ In case  $F$ is also  abelian,  then $A\times F$ is hereditarily totally minimal if and only if
 $F$ is either trivial or a $p$-group (see \cite{DS}).
\een
\end{example}

\begin{proposition}\label{prop:COM}
Let $G$ be a topological group. If there exists a compact open normal subgroup $N$ of $G$ such that the quotient group $G/N$ is hereditarily  (totally) minimal, then $G$ is c-(totally) minimal.
\end{proposition}
\begin{proof}
Let $H$ be a closed subgroup of $G$. Since $G/N$ is a  hereditarily (totally) minimal group, its subgroup $HN/N$ is (totally) minimal.
The  discrete groups $HN/N$  and $H/(H\cap N)$ are  isomorphic. So $H/(H\cap N)$ is also (totally) minimal group with $H\cap N$ compact.  By Fact \ref{fac:EDS}, $H$ is (totally) minimal.
\end{proof}

We now provide an example of locally compact \CTM\ groups that are neither compact nor discrete. It is still open whether there exists a locally compact \HM\ group that is  neither compact nor discrete (see \cite[Question 7.3(1)]{firstpaper}).  On the other hand, we shall see in Proposition \ref{cor:csol} below that a complete  solvable \CTM \ group must be compact.

\begin{example}\label{example:NHMCM}
Recall that a group is called { \em hereditarily non-topologizable} when it is hereditarily totally minimal in the discrete topology. It is known that such groups exist (see \cite{KOO}).	Let $G=N\times T$, where $N$ is a compact group and $T$ is a discrete hereditarily non-topologizable group. Then $G$ is \CTM\ by Proposition \ref{prop:COM}.
\end{example}

The following lemma will be used in the proof of Proposition \ref{cor:csol}.

\begin{lemma}\label{lem:referee}
Let $G_1, G_2$ be two subgroups of a topological group $G$.
\ben \item If  $G_1$ is  normal in $G_2$, then $\overline{G_1}$ is normal in $\overline{G_2}$, where $\overline{G_i}$ denotes the closure of $G_i$ in $G$.\item If, in addition, $G_2/G_1$ is abelian, then so is $\overline{G_2}/\overline{G_1}.$ \een
\end{lemma}
\begin{proof}
(1) Let us see that if $y\in \overline{G_2}$ and $x\in \overline{G_1},$ then $yxy^{-1}\in \overline{G_1}.$ Take nets $(x_i)_{i\in I}$ and $(y_i)_{i\in I}$  in $G_1$ and $G_2$, respectively,  such that $\lim y_i=y$ and $\lim x_i=x$. Then for every $i\in I$ we have  $y_ix_iy_i^{-1}\in G_1$ since $G_1$ is  normal in $G_2.$ As $yxy^{-1}=\lim y_ix_iy_i^{-1}$, we deduce that  $yxy^{-1}\in \overline{G_1}.$ \\
(2) It suffices to show that the commutator $[x,y]\in  \overline{G_1},$ whenever $x\in \overline{G_2}, y\in \overline{G_1}.$ If the  nets $(x_i)_{i\in I}\subseteq G_2$ and $(y_i)_{i\in I}\subseteq G_1$ converge to $x$ and $y$, respectively, then $[x_i,y_i]\in  G_1$ for every $i\in I$ since $G_2/G_1$ is abelian. Being the limit of $[x_i,y_i]$, the commutator $[x,y]$ belongs to  $\overline{G_1}.$
\end{proof}

\begin{proposition} \label{cor:csol}
If $G$ is a  solvable \CTM\ group, then $G$ is precompact. In particular, a complete solvable \CTM \ group is compact.
\end{proposition}
\begin{proof}
As $G$ is solvable,  there exist $k\in \N$ and a subnormal series $$G_0=\{e\}\unlhd G_1\unlhd \cdots \unlhd G_k=G$$ such that the quotient group $G_j/G_{j-1}$  is abelian for every $j\in \{1,\ldots, k\}.$  By Lemma \ref{lem:referee},  we may assume without loss of generality that the subgroups $G_j$ are closed in $G$.

Since $G_0$ is compact, it suffices to show that if  $G_{j-1}$ is precompact for some  $j\in \{1,\ldots, k\}$, then so is $G_j$.  As $G$ is \CTM,\ the abelian factor group $G_j/G_{j-1}$ is minimal. Using Fact \ref{fac:maispre}, we deduce that $G_j/G_{j-1}$ is precompact. As precompactness is a three space property (see \cite{BT} for example), the subgroup $G_j$ is also precompact, as needed.
\end{proof}

\begin{remark}  Dikarnjan and Megrelishvili asked (see  \cite[Question 3.10(b)]{DMe}) whether  solvable (at least metabelian) totally minimal groups must be precompact. Answering  this question in the negative, we provide in the next  example a metabelian totally minimal Lie group that is not compact. In particular,  it follows that  Proposition \ref{cor:csol} cannot be extended to solvable totally minimal groups. Note that a totally minimal nilpotent group is precompact (see \cite[Theorem 3.11]{DMe}).
\end{remark}

\begin{example}
Mayer \cite[Examples 2.6(i)]{Mayer}  proved  that the Euclidean motion group $\R^n \leftthreetimes \SO(n,\R)$ is totally minimal for every $n>1$, where   the special orthogonal group $\SO(n,\R)$ acts on $\R^n$ by matrix multiplication.
Fixing $n=2$ one obtains  the  Lie group of orientation-preserving isometries of the complex plane \[G:=\ISO_{+}(\C)= \Bigg\{\left(\begin{array}{cc}
a & b \\
0 & 1 \\
\end{array}\right)\bigg | \ \ a \in \T, b\in \C \Bigg\}\cong \C \leftthreetimes \T,\]  where $\T$ acts on $\C$ by multiplication. Note that the total minimality of the non-compact metabelian Lie group $G$ can be established in a different way by  Banakh \cite[Theorem 13]{Banakh1}.
\end{example}

The following lemma was proved in \cite[Lemma 6.19]{firstpaper} only for hereditarily totally minimal  groups. We provide a similar proof here for the sake of completeness.

\begin{lemma}\label{lemma:strong}
If $G$ is a \CTM\ group, then all quotients of  $G$ are \CTM.
\end{lemma}
\begin{proof}
Let $N$ be a closed normal subgroup of  $G$ and $q:G \to G/N$ be the quotient map. Take a closed subgroup $D$ of $G/N$ and let $D_1 = q^{-1}(D)$. Now we prove that $D$ is totally minimal. Consider the restriction $q': D_1 \to D$. Clearly, $q'$ is a continuous surjective  homomorphism. Since $D_1$ is totally minimal by the hypothesis on $G$, we obtain that $q'$ is open. Moreover, being a quotient group of $D_1$, we deduce that $D$ is totally minimal.
\end{proof}

\begin{theorem}\label{thm:ballie}
Let $N$ be a compact open normal subgroup of a topological group $G$.Then the following conditions are equivalent:
\ben
\item $G$ is c-totally minimal;
\item $G/N$ is hereditarily totally minimal (i.e., hereditarily  non-topologizable).
\een
\end{theorem}
\begin{proof}
$(1)\Rightarrow (2)$:  As $G$ is $c$-totally minimal, Lemma \ref{lemma:strong} implies that the quotient $G/N$ is c-totally minimal. Being discrete, $G/N$ is hereditarily totally minimal.\\
$(2)\Rightarrow (1)$:   See Proposition \ref{prop:COM}.
\end{proof}

\begin{corollary}\label{cor:balorlie}
Let $G$ be  a locally compact  group that is either balanced or Lie. Then the following conditions are equivalent:
\ben
\item  $G$ is c-totally minimal;
\item there exists a compact open normal subgroup $N$ of $G$ such that   $G/N$ is hereditarily totally minimal.\een
\end{corollary}
\begin{proof}
$(1)\Rightarrow (2)$: By Lemma \ref{lem:liebal}, there exists a compact open normal subgroup $N$ of $G$. Using Theorem \ref{thm:ballie}, we deduce that
$G/N$ is hereditarily totally minimal.\\
$(2)\Rightarrow (1)$: Use Theorem  \ref{thm:ballie}.
\end{proof}

We now provide another sufficient condition for the compactness of a c-totally minimal group.
A topological group $G$ is called  {\em strongly compactly covered} (see \cite{GBST}) if every element of $G$ is contained in a compact open normal subgroup of $G$.

\begin{corollary}
Let $G$ be a strongly compactly covered group. If $G$ is \CTM,\  then it is compact.
\end{corollary}
\begin{proof}
By \cite[Proposition 1.1]{GBST},  $G$  contains a compact open normal subgroup $N$ such that $G/N$ is a torsion FC-group (i.e., every element has finitely many conjugates).
By  Theorem \ref{thm:ballie}, this discrete locally finite group  is also hereditarily totally minimal. According to Fact \ref{fac:dhm}, which says that  a locally finite discrete hereditarily minimal group is finite, we have that $G/N$ must be finite. The finiteness of $G/N$ implies that $G$ is compact, as needed.
\end{proof}

\section{Open questions and concluding remarks}\label{sec:open}

In view of  Theorem \ref{pro:procmin} a natural question arises:
\begin{question}
	Let $A$ be a (totally) minimal abelian group and $K$ be a compact group. Is $G=A\times K$ c-(totally) minimal?
\end{question}

Note that by  Fact \ref{fac:EDS} $G$ is (totally) minimal whenever $A$ has the same property.

\vskip 0.3 cm

By Proposition \ref{cor:csol}, a complete \CTM \ solvable group is compact. Can this result be extended to locally solvable groups?  In other words:

\begin{question}\label{q:ex}
Let $G$ be a complete \CTM \ locally solvable group. Is $G$ compact?
\end{question}

The next proposition provides a positive answer to Question \ref{q:ex} in case $G$ has a compact open normal subgroup.
\begin{proposition}
Let $G$ be a  \CTM \ locally solvable group. If $G$ has a compact open normal subgroup $N$, then $G$ is compact.
\end{proposition}
\begin{proof}
By Theorem \ref{thm:ballie}, the locally solvable discrete group $G/N$ is hereditarily totally minimal. By Fact \ref{fac:dhm}, $G/N$ is finite.
This implies that $G$ is compact.
\end{proof}

We proved in Theorem \ref{comp} that a locally solvable $c$-minimal Lie group is compact. In view of Lemma \ref{lem:liebal} we ask:
\begin{question}\label{que:lievsba}
Let $G$ be a locally compact, locally solvable, $c$-minimal balanced group. Is $G$ compact?
\end{question}

Let $G$ be a locally compact balanced group satisfying the following conditions:
\ben
\item every abelian subgroup of $G$ has compact closure;
\item $G$ is topologically locally finite (i.e., each precompact subset of $G$ generates a precompact subgroup).
\een

Grosser and Herfort conjectured that such a topological group is either totally disconnected or compact (see Remark on page 222 of \cite{GH}).
Even though we cannot provide a positive answer to Question \ref{que:lievsba} at this point,  we shall see in Corollary \ref{cor:last} below that a locally compact, locally solvable, $c$-minimal balanced group satisfies the compactness conditions (1)-(2). The next proposition extends \cite[Proposition 2.2]{DGB} from the abelian case.

\begin{proposition}\label{pro:tlf}
Let $G$ be a compactly covered, locally compact, locally solvable balanced group. Then every compact subset of $G$ is contained in a compact open subgroup of $G$.
\end{proposition}
\begin{proof}
Let $K$ be a compact subset of $G$. Since $G$ is locally compact there exists a  compact neighborhood $V$ of the identity $e$. Replacing $K,$  if necessary, with its superset $(K\cup\{e\})\cdot V$ we may assume, without loss of generality, that $K$ itself is a neighborhood of $e$. The rest of the proof is  very similar to the proof of \cite[Proposition 2.2]{DGB},  and we give it here for the sake of completeness.

Consider the quotient homomorphism $q: G \to G/c(G)$.  As $G$ is a locally compact compactly covered group, so is its closed subgroup $c(G)$. By Fact \ref{fac:cc}, $c(G)$ is compact. Since $G/c(G)$ is  a locally compact totally disconnected balanced group, it has a local base at the identity consisting of compact open normal subgroups. Let $U$ be a compact open normal subgroup of $G/c(G)$ contained in $q(K).$ Being discrete and compactly covered, the group $(G/c(G))/U$ is torsion.
Using \cite[Proposition 1.1.5]{DIX},we have that the locally solvable group $(G/c(G))/U$ is locally finite.
Hence, there exists a finite subgroup $N$ of $(G/c(G))/U$ containing  the finite set $r(q(K))$, where $$r:G/c(G)\to (G/c(G))/U$$ is the quotient homomorphism. Since the groups $U,c(G)$ and $N$ are all compact, it follows that   $(r\circ q)^{-1}(N)$ is a compact open subgroup of $G$ containing $K$, as needed.
\end{proof}

\begin{corollary}\label{cor:last}
Let $G$ be a locally compact, locally solvable, $c$-minimal balanced group. Then $G$ is topologically locally finite and  every abelian subgroup of $G$ has compact closure.
\end{corollary}
\begin{proof}
By Lemma \ref{lem:ccov}, $G$ is compactly covered. It follows from Proposition \ref{pro:tlf} that $G$ is topologically locally finite.
By Fact \ref{fac:maispre}, every  abelian subgroup of $G$ has compact closure.
\end{proof}

Let $N$ be a compact open normal subgroup of a topological group $G$. By Proposition \ref{prop:COM}, if $G/N$ is hereditarily  minimal, then $G$ is c-minimal. Moreover, by Theorem \ref{thm:ballie}, $G$ is \CTM\ if and only if $G/N$ is hereditarily totally minimal.

\begin{question}
Let $G$ be a c-minimal  group with a compact open normal subgroup $N$. Is $G/N$ hereditarily  minimal?
\end{question}

 \subsection*{Acknowledgments}
It is a pleasure to thank the referee for a very careful reading and a wealth of helpful comments that improved this paper substantially. We thank M. Megrelishvili whose question motivated  this research. Thanks are due also to D. Dikranjan and D. Toller for their useful suggestions.

The first author is supported by the National Natural Science Foundation of China (Grant No. 12001264) and the Natural Science Foundation of Jiangsu Province (Grant No. BK20200834).

\end{document}